\documentclass{amsart}

\usepackage{amssymb}
\usepackage{amsfonts}
\usepackage{amsmath}
\usepackage{amsthm}
\usepackage{graphicx}
\usepackage{mathrsfs}
\usepackage{dsfont}
\usepackage{amscd}
\usepackage{multirow}
\usepackage{palatino}
\usepackage{mathpazo}
\usepackage{mathtools}
\usepackage[all]{xy}
\usepackage[scaled]{helvet} 
\usepackage{courier} 
\normalfont
\usepackage[T1]{fontenc}
\usepackage{calligra}
\usepackage{verbatim}
\usepackage[usenames,dvipsnames]{color}
\usepackage[colorlinks=true,linkcolor=Blue,citecolor=Violet]{hyperref}
\usepackage{cite}
\usepackage{enumitem}

\makeatletter
\@namedef{subjclassname@2010}{%
  \textup{2010} Mathematics Subject Classification}
\makeatother

\theoremstyle{plain}
\newtheorem{theorem}{Theorem}[section]

\newtheorem{lemma}[theorem]{Lemma}
\newtheorem{proposition}[theorem]{Proposition}

\newtheorem{theorem-definition}[theorem]{Theorem-Definition}

\theoremstyle{definition}
\newtheorem{definition}[theorem]{Definition}

\theoremstyle{remark}

\numberwithin{equation}{section}
 
\newcommand{\dd}{{\mathsf{d}}}
\newcommand{\N}{{\mathds{N}}}

\newcommand{\Q}{{\mathds{Q}}}
\newcommand{\R}{{\mathds{R}}}
\newcommand{\C}{{\mathds{C}}}

\newcommand{\A}{{\mathcal{A}}}
\newcommand{\B}{{\mathcal{B}}}
\newcommand{\M}{{\mathfrak{M}}}

\newcommand{\qpropinquity}[1]{{\mathsf{\Lambda}_{#1}}}

\newcommand{\BaireSpace}{{\mathscr{N}}}

\newcommand{\diag}{\mathrm{diag}}

\renewcommand{\geq}{\geqslant}
\renewcommand{\leq}{\leqslant}

\allowdisplaybreaks[1]

\makeatletter
\newcommand{\vast}{\bBigg@{4}}
\newcommand{\Vast}{\bBigg@{5}}
\makeatother

\begin{document}
  
\title[ Gromov-Hausdorff propinquity  and Christensen-Ivan quantum metrics]{quantum Gromov-Hausdorff propinquity convergence of Christensen-Ivan quantum metrics on AF algebras}
\author{Clay Adams}

\author{Konrad Aguilar}
\address{Department of Mathematics and Statistics, Pomona College, 610 N. College Ave., Claremont, CA 91711} 
\email{konrad.aguilar@pomona.edu}
\urladdr{\url{https://aguilar.sites.pomona.edu}}
\thanks{The second author is supported by NSF grant DMS-2316892}

\author{Esteban Ayala}

\author{Evelyne Knight}

\author{Chloe Marple}

\date{\today}

\begin{abstract}
 We provide convergence in the quantum Gromov-Hausdorff propinquity of Latr\'emoli\`ere of some sequences of infinite-dimensional Leibniz compact quantum metric spaces of Rieffel given by AF algebras and Christensen-Ivan spectral spaces. The main examples are convergence of Effros-Shen algebras and UHF algebras.
\end{abstract}
\maketitle
\tableofcontents

\section{Introduction and Background}

The first example of convergence of sequences of infinite-dimensional quantum metric spaces was   established by Rieffel \cite{Rieffel00}, where he showed that the quantum tori converged with respect their parameters that defined their anti-commutation relation. This was accomplished by the introduction of the theory of quantum metric spaces and a noncommutative analogue to the Gromov-Hausdorff distance both introduce by Rieffel in \cite{Rieffel98a, Rieffel00}, respectively. This introduced an new field of study known as noncommutative metric geometry, which has its roots from work of Connes in \cite{Connes, Connes89} and Gromov \cite{Gromov81}.

Since the introduction of Rieffel's noncommutative analogue to the Gromov-Hausdorff distance, there has been much progress in developing noncommutative analogues of the Gromov-Hausdorff distance to capture the C*-algebraic structure of the quantum metric space \cite{li03, Kerr02, wu06a, Latremoliere13, Latremoliere13b}. In particular,   in \cite{Latremoliere13c},  Latr\'emoli\`ere proved convergence of the quantum tori in this stronger sense. Moreoever, 
the Gromov-Hausdorff propinquity of Latr\'emoli\`ere first introduced in \cite{Latremoliere13} has been adapted to capture more structure such as module structure \cite{Latremoliere16} and spectral triple structure \cite{Latremoliere22a, Latremoliere22b}.

Another example of convergence of infinite-dimensional quantum metric spaces appeared in \cite{Aguilar-Latremoliere15}, where it was shown that the Effros-Shen algebras \cite{Effros80b} are continuous in Gromov-Hausdorff propinquity with respect to their natural parameter space of irrationals in $(0,1)$ with the usuual topology. It was also show that UHF algebras are continuous with respect to their natural parameter space of multiplicity sequences metrized by the Baire space. This was accomplished by introducing new quantum metrics on AF algebras equipped with faithful tracial state motivated by work of Christensen and Ivan in \cite{Christensen06}. Now, in \cite{Christensen06}, Christensen and Ivan did introduce quantum metrics on these infinite-dimensional algebras, but at the time it wasn't clear how to provide convergence of these algebras in any noncommutative analogue to the Gromov-Hausdorff distance, which is one reason why the quantum metrics of \cite{Aguilar-Latremoliere15} were introduced which are not defined using spectral triples. But, in an effort, to bring the realms of noncommutative geometry and noncommutative metric geometry closer, it is important to provide the convergence results of these infinite-dimensional algebras of \cite{Aguilar-Latremoliere15} with the quantum metrics induced by the spectral triples of \cite{Christensen06}, which is exactly what is accomplished in this article. 

The main hurdles to overcome in proving this arise from two issues that were circumvented by the quantum metrics introduced in \cite{Aguilar-Latremoliere15}. First, the spectral triples of \cite{Christensen06} are constructed using equivalence constants which are only provided by existence and not explicitly given, which cause an issue when providing continuous fields of $L$-seminorms as it is difficult to control these non-explicit constants. Second, providing continuous fields of $L$-seminorms provided by faithful tracial states is difficult when relying on convergence in various operator norms  given by different GNS represenations for each spectral triple rather than a fixed C*-norm. The first issue is overcome by an application of \cite[Lemma 3.9]{Aguilar24}, and the second issue is overcome by a generalization of \cite[Lemma 3.9]{Aguilar24}. Both of these issues are overcome in Section \ref{s:cond-fd}, and we apply these results in the last section to provide convergence of these infinite-dimensional algebras using quantum metrics induced by Christensen-Ivan spectral triples. We only define what we mean by a Leibniz compact quantum metric space as things can get quite overwhelming as more definitions are provided, but for references regarding quantum metric spaces, propinquity and propinquity in the context of AF algebras see \cite{Rieffel05, Latremoliere13, Aguilar-Latremoliere15, Latremoliere15b}.

\begin{definition}{\cite{Rieffel98a,Latremoliere13}}
    Let $\A$ be a unital C*-algebra with norm $\|\cdot\|_\A$ and unit $1_\A$. Let $L:\A \rightarrow [0,\infty)$ be a seminnorm (possibly taking value $\infty$) such that $\mathrm{dom}(L)=\{ a\in \A : L(a) <\infty\}$ is a dense *-subalgebra of $\A$. If
    \begin{enumerate}
        \item $L(a)=L(a^*)$ for every $a \in \A$,
        \item $\{a \in \A : L(a)=0\}=\C1_\A$,
        \item $L(ab)\leq \|a\|_\A L(b) +\|b\|_\A L(a)$ for every $a,b \in \A$,
        \item the metric on the state space $S(\A)$  of $\A$ defined for every $\phi, \psi\in S(\A)$ by
        \[
        mk_L(\phi, \psi)=\sup \{|\phi(a)-\psi(a)| : a \in \A, L(a) \leq 1\}
        \]
        metrizes the weak* topology,
    \end{enumerate}
    then we call $L$ an {\em $L$-seminorm} and $(\A, L)$ a {\em Leibniz compact quantum metric space}.
\end{definition}

\section{Finite-dimensional approximations and associated continuous fields of $L$-seminorms}\label{s:cond-fd}

In what follows, we use various results from the beginning of \cite[Section 2]{Christensen06} with some slightly different notation. 
 $\A=\overline{\cup_{n \in \N} A_n}^{\|\cdot\|_A}$ be a unital AF algebra, where $\A_0=\C1_\A$ equipped with a faithful tracial state $\tau$. Let $H_\tau$ denote the associated GNS Hilbert space with inner product defined for every $a,b \in H_\tau$ by \[
 \langle a,b\rangle_\tau =\tau(b^*a) 
 \]
 and associated norm $\|a\|_\tau=\sqrt{\langle a,a\rangle_\tau}$.  Since $\tau$ is faithful, we can canonically view $\A$ as a subspace (not necessarily closed) of $H_\tau$. Let 
 \[
 \pi_\tau : \A \longrightarrow B(H_\tau)
 \] be the associated GNS representation such that $\pi_\tau(a)(b)=ab$ for every $a,b \in \A$.

Let $n \in \N$, since $\A_n$ is finite dimensional, we have that $\A_n$ is a closed subspace of $H_\tau$. Let 
 \[
 P^\tau_n: H_\tau\rightarrow \A_n
 \]
 denote the orthogonal projection of $H_\tau$ onto $\A_n$ and define $Q^\tau_n=P^\tau_n-P^\tau_{n-1}$. Let 
 \[
 E^\tau_n: \A \rightarrow \A_n
 \]
 denote the restriction of $P^\tau_n$ to $\A$, and by \cite[Theorem 3.5]{Aguilar-Latremoliere15}, we have that $E^\tau_n$ is the unique $\tau$-preserving conditional expectation onto $\A_n$.

 Next, since $\A_{n+1}$ is finite dimensional, there exists a sharp $c^\tau_{n+1}>0 $ such that  
 \begin{equation}\label{eq:sharp-constants}
 \|a\|_\A \leq  c^\tau_{n+1} \cdot \|a\|_\tau
 \end{equation} for every $ a\in \A_{n+1}$.  Note that $c^\tau_{n+1}\geq 1$ since $\|\cdot\|_\tau \leq \|\cdot\|_\A$ on $\A$.

 We now prove a crucial fact about these constants.
 \begin{proposition}\label{p:constant-conv}
     Let $(\tau^n)_{n \in \N}$ be a sequence of faithful tracial states on $\A$ and let $\tau$ be a faithful tracial state on $\A$. If $(\tau^n)_{n \in \N}$ converges to $\tau$ in the weak* topology, then for every $N \in \N$, the sequence $(c^{\tau^n}_{N})_{n \in \N}$ converges to $c^{\tau}_{N}$ in the usual topology on $\R$. 
 \end{proposition}
 \begin{proof}
     This is just \cite[Proposition 3.10]{Aguilar24} applied to \cite[Proposition 3.6]{Aguilar24} since norm $\|\cdot\|_\tau$ is a Frobenius-Rieffel norm.
 \end{proof}

 Let $(\beta(n))_{n \in \N}$ be a summable sequence of positive reals.  Set
 \[
 a^\tau_{\beta, n+1}=\frac{c^\tau_{n+1}}{\beta_{n+1}}.
 \]

Next, we state a main result from \cite{Christensen06}.
\begin{theorem}{\! \cite[Theorem 2.1]{Christensen06}}\label{t:CI-thm}
     Let $(\beta(n))_{n \in \N}$ be a summable sequence of positive reals. Using the above setting, we have that 
     \[
     D^\tau_\beta=\sum_{n=1}^\infty a^\tau_{\beta,n} Q^\tau_n
     \]
     defines an unbounded self-adjoint operator on $H_\tau$. Furthermore, if we define 
     \[
     L^\tau_\beta(a)=\|[D^\tau_\beta, \pi_\tau(a)]\|_{B(H_\tau)}
     \]
     for every $a\in \A$ such that $[D^\tau_\beta, \pi_\tau(a)]$ extends to a bounded operator on $H_\tau$ denoted by $[D^\tau_\beta, \pi_\tau(a)]$, and set $ L^\tau_\beta(a)=\infty$ if not, then
     \[
     (\A, L^\tau_\beta)
     \]
     is a Leibniz compact quantum metric space, and for every $n\in \N$, $(\A_n, L^\tau_\beta)$ is a Leibniz compact quantum metric space such that  $\mathrm{dom}(L^\tau_\beta(a))\cap\A_n=\A_n$.
\end{theorem}

The following fact is stated after  \cite[Expression (4.5)]{Latremoliere23}, but we provide a proof here.

\begin{proposition}\label{p:f-dim-comm}
    Using the setting of Theorem \ref{t:CI-thm}, we have for every $n \in \N$ and for every $a \in \A_n$  that 
    \[
    L^\tau_\beta(a)=\|[D^\tau_\beta, \pi_\tau(a)]\|_{B(\A^\tau_n)},
    \]
    where $\A_n^\tau=\A_n$ but for $B(\A^\tau_n)$ we are considering bounded operators with respect to the norm $\|\cdot\|_\tau$ on $\A_n$.
\end{proposition}
\begin{proof}
    Let $n \in \N$. Let $a \in \A_n$. By definition, we have that 
    \[
    \|[D^\tau_\beta, \pi_\tau(a)]\|_{B(\A_n)}\leq L^\tau_\beta(a).
    \]
    Next, let  $k \in \{1,2,\ldots, n\}$. We have since $\pi_\tau(a)$ commutes with $P^\tau_n$ by the proof of \cite[Theorem 2.1]{Christensen06} or the proof of Step 1 of \cite[Theorem 3.5]{Aguilar-Latremoliere15}. Moreover, $P_nP_k=P_kP_n=P_k$ and $P_nP_{k-1}=P_{k-1}P_n=P_{k-1}$ by construction. Thus
    \begin{align*}
    P^\tau_n[Q^\tau_k,\pi_\tau(a)]P^\tau_n& = P^\tau_n(Q^\tau_k\pi_\tau(a)-\pi_\tau(a)  Q^\tau_k)P^\tau_n\\
    & = P^\tau_n((P^\tau_k-P^\tau_{k-1})\pi_\tau(a)-\pi_\tau(a) (P^\tau_k-P^\tau_{k-1}))P^\tau_n\\
    & = (P^\tau_k-P^\tau_{k-1})\pi_\tau(a)P^\tau_n-P^\tau_n\pi_\tau(a)  (P^\tau_k-P^\tau_{k-1})\\
    & = (P^\tau_k-P^\tau_{k-1})P^\tau_n\pi_\tau(a)-\pi_\tau(a) P^\tau_n (P^\tau_k-P^\tau_{k-1})\\
    & = (P^\tau_k-P^\tau_{k-1})\pi_\tau(a)-\pi_\tau(a) (P^\tau_k-P^\tau_{k-1})\\
    & = Q^\tau_k\pi_\tau(a)-\pi_\tau(a)  Q^\tau_k\\
    & = [Q^\tau_k,\pi_\tau(a)]
\end{align*}    
Thus
    \begin{align*}
        [D^\tau_\beta, \pi_\tau(a)]& = \sum_{k=1}^n a^\tau_{\beta, k}[Q^\tau_k,\pi_\tau(a)]\\
        & = \sum_{k=1}^n a^\tau_{\beta, k}P^\tau_n[Q^\tau_k,\pi_\tau(a)]P^\tau_n\\
        & =P^\tau_n [D^\tau_\beta, \pi_\tau(a)]P^\tau_n.
    \end{align*}
    Hence, since $(P^\tau_n)^2=P^\tau_n$ and $P^\tau_n$ is contractive with respect to $\|\cdot\|_\tau$ and $P^\tau_n(h) \in \A_n$ for every $ h\in H_\tau$, we have 
    \begin{align*}
        L^\tau_\beta(a)& =\| P^\tau_n [D^\tau_\beta, \pi_\tau(a)]P^\tau_n \|_{B(H_\tau)}\\
        & = \sup \left\{\| P^\tau_n [D^\tau_\beta, \pi_\tau(a)]P^\tau_n (h)\|_\tau : h \in H_\tau, \|h\|_\tau \leq 1 \right\}\\
        & =  \sup \left\{\| P^\tau_n [D^\tau_\beta, \pi_\tau(a)](P^\tau_n)^2 (h)\|_\tau : h \in H_\tau, \|h\|_\tau \leq 1 \right\}\\
        & =  \sup \left\{\| P^\tau_n [D^\tau_\beta, \pi_\tau(a)]P^\tau_n(P^\tau_n (h))\|_\tau : h \in H_\tau, \|h\|_\tau \leq 1 \right\}\\
        & =  \sup \left\{\|  [D^\tau_\beta, \pi_\tau(a)] (P^\tau_n (h))\|_\tau : h \in H_\tau, \|h\|_\tau \leq 1 \right\}\\
        & \leq  \sup \left\{\|  [D^\tau_\beta, \pi_\tau(a)] \|_\tau : h \in \A_n, \|h\|_\tau \leq 1 \right\}\\
        & = \|[D^\tau_\beta, \pi_\tau(a)]\|_{B(\A_n)}.
    \end{align*}
Therefore $ \|[D^\tau_\beta, \pi_\tau(a)]\|_{B(\A_n)}\leq L^\tau_\beta(a)\leq \|[D^\tau_\beta, \pi_\tau(a)]\|_{B(\A_n)}$ as desired.
\end{proof}
With this we can provide finite-dimensional approximations, which has been conveniently already proven   in  \cite{Latremoliere23}.

\begin{theorem}{\!  \cite[Theorem 4.8]{Latremoliere23}}\label{t:fd-approx}
    For every $n \in \N$, it holds that 
    \[
    \qpropinquity{}((\A, L^\tau_\beta), (\A_n, L^\tau_\beta))\leq \sum_{k=n}^\infty \beta_k,
    \]
    where $\qpropinquity{}$ is the quantum Gromov-Hausdorff propinquity of \cite{Latremoliere13}.
\end{theorem}

The main examples of AF algebras in this article, Effros-Shen algebras and UHF algebras, are given in the setting of inductive limits of finite-dimensional C*-algebras. Thus, we introduce notation to prove results in this setting.

Let $(\B_n, \alpha_n)_{n \in \N}$ be an inductive sequence of C*-algebras (see \cite[Section 6.1]{Murphy90}) such that:
\begin{enumerate} 
\item $\B_0=\C$ and $\B_n=\bigoplus_{k=1}^{n_n} \M_{d_{n,k}}(\C)$ for all $n \in \N\setminus \{0\}$, where $d_{n,k} \in \N \setminus \{0\}$ for each $n \in \N \setminus \{0\}$ and $k \in \{1, 2, \ldots, n_n\}$;
\item   $\alpha_n :\B_n \rightarrow \B_{n+1}$ is a unital *-monomorphism for all $n \in \N$;
\item   the inductive limit $\A=\underrightarrow{\lim} \ (\B_n, \alpha_n)_{n \in \N}$ is equipped with a faithful tracial state $\tau$.
  \end{enumerate}
  For each $n \in \N$, let $  \alpha^{(n)} : \B_n \rightarrow \A$ be the canonical unital *-monomorphism satisfying 
  \begin{equation} 
  \alpha^{(n+1)}\circ \alpha_n=\alpha^{(n)},
  \end{equation}
  and if for each $k\in \{1,2,\ldots, n-1\}$, we define
  \[
  \alpha_{k,n}=\alpha_n\circ \alpha_{n-1} \circ \cdots \alpha_k,
  \] then inductively, we have  
  \begin{equation}\label{eq:ind-lim-maps}
       \alpha^{(n+1)}\circ \alpha_{k,n}=\alpha^{(k)}
  \end{equation}
  
  Note that $\A=\overline{\cup_{n \in \N}\alpha^{(n)}(\B_n)}^{\|\cdot\|_\A}$ and $\alpha^{(n)}(\B_n) \subseteq \alpha^{(n+1)}(\B_{n+1})$ and $\alpha^{(0)}(\B_0)=\C1_\A$ (see \cite[Section 6.1]{Murphy90}). So, for each $n \in \N$,
  set
  \[
  \A_n=\alpha^{(n)}(\B_n).
  \]
  As above, for each $n \in \N$, let \[E^\tau_n :\A \rightarrow  \A_n\]
  denote the unique $\tau$-preserving faithful conditional expectation onto $\A_n.$ For each $n \in \N$, let
\begin{equation}\label{eq:fd-trace}
\tau_n=\tau \circ \alpha^{(n)},
\end{equation}
which is a faithful tracial state on  $\B_n$ and let $\pi_{\tau_n}$ denote the associated GNS representation. Let  $k \in \{0, 1, \ldots, n+1\}$ let \[E^{\tau_{n+1}}_{n+1,k}: \B_{n+1} \rightarrow \alpha_{k,n}(\B_k)\]
be the unique $\tau_{n+1}$-preserving faithful conditional expectation onto $\alpha_{k,n}(\B_k)$. Define
\[
Q^{\tau_{n+1}}_{n+1,k}=E^{\tau_{n+1}}_{n+1,k}-E^{\tau_{n+1}}_{n+1,k-1}
\]
and let 
\[
D^{\tau_{n+1}}_\beta=\sum_{k=1}^{n+1} a^\tau_{\beta,k}Q^{\tau_{n+1}}_{n+1,k}.
\]
For every $a \in \B_{n+1}$, define
\begin{equation}\label{eq:ind-lim-fd-lip}
L^{\tau_{n+1}}_\beta(a)=\|[D^{\tau_{n+1}}_\beta, \pi_{\tau_{n+1}}(a)]\|_{B(\B^{\tau_{n+1}}_{n+1})}.
\end{equation}
By finite dimensionality, we have that 
\[
(\B_{n+1},L^{\tau_{n+1}}_\beta )
\]
is a Leibniz compact quantum metric space. We can now add to Proposition \ref{p:f-dim-comm} in the inductive limit setting.

\begin{theorem}\label{t:ind-lim-comm}
    Let $n \in \N$. It holds that \[
    L^\tau_\beta\circ \alpha^{(n)}(a)= \sup \{ \|[D^\tau_\beta, \pi_\tau(\alpha^{(n)}(a))](\alpha^{(n)}(b))  \|_\tau: b \in \B_n, \|b\|_{\tau_n} \leq 1\}=L^{\tau_n}_\beta(a)
    \]
    for every $a \in \B_n$.
\end{theorem}

\begin{proof}
 Let $n \in \N$ and let $a \in \B_n$, then by Proposition \ref{p:f-dim-comm}
 \begin{align*}
 L^\tau_\beta\circ \alpha^{(n)}(a)& = L^\tau_\beta(\alpha^{(n)}(a))\\
 &= \|[D^\tau_\beta, \pi_\tau(\alpha^{(n)}(a))]\|_{B(\A^\tau_n)}\\
 & = \sup \{ \|[D^\tau_\beta, \pi_\tau(\alpha^{(n)}(a))](c)  \|_\tau: c \in \A_n, \|c\|_\tau \leq 1\}.
 \end{align*}
 Consider $c\in \A_n$, then there exists a unique $b \in \B_n$ such that $\alpha^{(n)}(b)=c$. We have
 \begin{align*}
 \|c\|_\tau^2=\tau(c^*c)=\tau\left(\alpha^{(n)}(b)^*\alpha^{(n)}(b) \right)=\tau \left( \alpha^{(n)}(b^*b)\right)= \tau_n(b^*b)=\|b\|^2_{\tau_n}.
 \end{align*}
 Hence
\[
L^\tau_\beta\circ \alpha^{(n)}(a)= \sup \{ \|[D^\tau_\beta, \pi_\tau(\alpha^{(n)}(a))](\alpha^{(n)}(b))  \|_\tau: b \in \B_n, \|b\|_{\tau_n} \leq 1\}
\]
Let $b \in \B_n$. 
Let $k \in \{0,1,\ldots, n\}$.   Then a similar argument as the beginning of the proof of \cite[Proposition 3.5]{Aguilar24} provides\[ 
E^\tau_k\circ \alpha^{(n)}   = \alpha^{(n)}\circ E^{\tau_n}_{n,k} 
\]
and
\[
E^\tau_{k-1}\circ \alpha^{(n)}   = \alpha^{(n)}\circ E^{\tau_n}_{n,k-1}
\] 
by Expression \eqref{eq:ind-lim-maps}. 
Next, we have 
\begin{align*}
& [Q^\tau_k,\pi_\tau(\alpha^{(n)}(a))](\alpha^{(n)}(b))\\
&\quad  =((E^\tau_k-E^\tau_{k-1})\pi_\tau(\alpha^{(n)}(a))-\pi_\tau(\alpha^{(n)}(a))(E^\tau_k-E^\tau_{k-1}))(\alpha^{(n)}(b)) 
\end{align*}
Now
\begin{align*}
\pi_\tau(\alpha^{(n)}(a))(E^\tau_k-E^\tau_{k-1})(\alpha^{(n)}(b))& =\pi_\tau(\alpha^{(n)}(a))(\alpha^{(n)}(E^{\tau_n}_{n,k}(b))-\alpha^{(n)}(E^{\tau_n}_{n,k-1}(b))) \\
& = \alpha^{(n)}(a)(\alpha^{(n)}(E^{\tau_n}_{n,k}(b))-\alpha^{(n)}(E^{\tau_n}_{n,k-1}(b)))\\
& = \alpha^{(n)}(aE^{\tau_n}_{n,k}(b)-aE^{\tau_n}_{n,k-1}(b))
\end{align*}
and similarly
\[
(E^\tau_k-E^\tau_{k-1})\pi_\tau(\alpha^{(n)}(a))(\alpha^{(n)}(b)) = \alpha^{(n)}(E^{\tau_n}_{n,k}(ab)-E^{\tau_n}_{n,k-1}(ab)).
\]
Thus
\begin{align*}
& [Q^\tau_k,\pi_\tau(\alpha^{(n)}(a))](\alpha^{(n)}(b))\\
&\quad  =\alpha^{(n)}\left(  E^{\tau_n}_{n,k}(ab)-E^{\tau_n}_{n,k-1}(ab)-(aE^{\tau_n}_{n,k}(b)-aE^{\tau_n}_{n,k-1}(b))\right).
\end{align*}
However, 
\begin{align*}
     & E^{\tau_n}_{n,k}(ab)-E^{\tau_n}_{n,k-1}(ab)-(aE^{\tau_n}_{n,k}(b)-aE^{\tau_n}_{n,k-1}(b)) \\
    & \quad = E^{\tau_n}_{n,k}(\pi_{\tau_n}(a)(b))-E^{\tau_n}_{n,k-1}(\pi_{\tau_n}(a)(b))\\
    & \quad \quad - (\pi_{\tau_n}(a)(E^{\tau_n}_{n,k}(b))-\pi_{\tau_n}(a)(E^{\tau_n}_{n,k-1}(b)))\\
    & \quad = (E^{\tau_n}_{n,k} -E^{\tau_n}_{n,k-1})(\pi_{\tau_n}(a)(b))- \pi_{\tau_n}(a) ((E^{\tau_n}_{n,k}- E^{\tau_n}_{n,k-1})(b))\\
    & \quad = Q^{\tau_n}_{n,k}(\pi_{\tau_n}(a)(b))- \pi_{\tau_n}(a)(( Q^{\tau_n}_{n,k})(b))\\
    &\quad  = (Q^{\tau_n}_{n,k}(\pi_{\tau_n}(a))- \pi_{\tau_n}(a)( Q^{\tau_n}_{n,k}))(b)\\
    & \quad =[Q^{\tau_n}_{n,k},\pi_{\tau_n}(a)](b).
\end{align*}
Hence
\[
[D^\tau_\beta, \pi_\tau(\alpha^{(n)}(a))](\alpha^{(n)}(b)) =\alpha^{(n)}([D^{\tau_n}_\beta,\pi_{\tau_n}(a) ](b))
\]
and so as above
\[
\|[D^\tau_\beta, \pi_\tau(\alpha^{(n)}(a))](\alpha^{(n)}(b))\|_\tau=\|[D^{\tau_n}_\beta,\pi_{\tau_n}(a) ](b)\|_{\tau_n}.
\]
Therefore
\[
L^\tau_\beta\circ \alpha^{(n)}(a)= \sup \{ \|[D^\tau_\beta, \pi_\tau(\alpha^{(n)}(a))](\alpha^{(n)}(b))  \|_\tau: b \in \B_n, \|b\|_{\tau_n} \leq 1\}=L^{\tau_n}_\beta(a)
\]
of Expression \eqref{eq:ind-lim-fd-lip}
as desired.
\end{proof}

 Now that we have an expression for the L-seminorms on the terms of a given inductive sequence, we would like to show that these form a continuous field of L-seminorms with respect to weak* convergence of the faithful tracial state. However, since the norms defining our L-seminorms are operator norms this takes some care, which is why we need some tools from metric geometry. The following result might be known in metric geometry, but we cannot find a proof and so we provide one here. The following result also serves as a generalization of \cite[Lemma 3.9]{Aguilar24}.

 \begin{lemma}\label{l:sup-conv}
 Let $(X, d)$ be  a metric space. Let $(C_n)_{n \in \N}$ be a sequence of compact subsets of $X$ that converges in the Hausdorff distance with respect to $d$, $Haus_d,$ to a compact $C\subseteq X$. Let $C'\subseteq X$ be a compact set such that $C\cup (\cup_{n \in \N} C_n)\subseteq C'$.    Let $(f_n)_{n \in \N}$ be a sequence of real-valued continuous functions on $X$ and let $f:X \rightarrow \R$ be continuous. 
 
 If $(f_n)_{n \in \N}$ converges uniformly to $f$  on $C'$, then    $(\sup_{x \in C_n} f_n(x))_{n \in \N}$ converges to $\sup_{x \in C} f(x)$ in the usual topology on $\R.$
 \end{lemma}
 \begin{proof} 
Let $\varepsilon>0$. By uniform convergence, there exists $\delta>0$ such that for every $a,b \in C'$ and $n \in \N$, we have 
\[
|f_n(a)-f_n(b)|<\varepsilon/2.
\]
Let $N \in \N$ such that for every $n \geq N$
\[
Haus_d(C_n,C)<\delta/3
\]
and  \[|\sup_{x \in C}f_n(x) -\sup_{x \in C}f(x)| < \varepsilon/2\] by \cite[Lemma 3.9]{Aguilar24}.

Let $n \geq N$. By compact, there exists $x' \in C$ such that 
\[
\sup_{x\in C}f_n(x)=f_n(x').
\]
Now consider $\sup_{x \in C_n}f_n(x).$ Assume by way of contradiction that $|\sup_{x \in C}f_n(x)-\sup_{x \in C_n}f_n(x)|>\varepsilon/2.$ Assume first that 
\[
\sup_{x \in C}f_n(x)-\sup_{x \in C_n}f_n(x)>\varepsilon/2
\]
\[
f_n(x')-\sup_{x \in C_n}f_n(x)>\varepsilon/2
\]
\[
\sup_{x \in C_n}f_n(x)<f_n(x')-\varepsilon/2.
\]
Hence
\[
f_n(x)<f_n(x')-\varepsilon/2
\]
for every $x \in C_n$. Now there exists $x\in C_n$ such that $\dd(x,x')<\delta$ by definition of the Hausdorff distance. Hence
\[
|f_n(x)-f_n(x')|<\varepsilon/2.
\]
And so 
\[
f_n(x')-\varepsilon/2<f_n(x)<f_n(x')-\varepsilon/2, 
\]
contradiction.

On the other hand, if $\sup_{x \in C_n}f_n(x)-\sup_{x \in C}f_n(x)>\varepsilon/2$. Then
\[
\varepsilon/2<\sup_{x \in C_n}f_n(x)-f_n(x').
\]
Now, by compact, there exists $z\in C_n$ such that $\sup_{x \in C_n}f_n(x)=f_n(z)$. Hence
\[
f_n(x')<f_n(z)-\varepsilon/2
\]
and so
\[
\sup_{x \in C}f_n(x)<f_n(z)-\varepsilon/2.
\]
Thus
\[
f_n(x)<f_n(z)-\varepsilon/2
\]
for every $x \in C$. This leads to a similar contradiction. Hence,
\[
|\sup_{x \in C_n}f_n(x) - \sup_{x \in C}f_n(x)|\leq \varepsilon/2.
\]
Finally,
\begin{align*}
    |\sup_{x \in C_n} f_n(x)-\sup_{x \in C} f(x)|& \leq |\sup_{x \in C_n} f_n(x)-\sup_{x \in C} f_n(x)|+|\sup_{x \in C} f_n(x)-\sup_{x \in C} f(x)|\\
    & \leq \varepsilon/2+|\sup_{x \in C} f_n(x)-\sup_{x \in C} f(x)|<\varepsilon/2+\varepsilon/2=\varepsilon. \qedhere
\end{align*}
 \end{proof}

 Before providing continuous fields of $L$-seminorms, we  need one more result so that we can satisfy the hypothesis of the previous Lemma.

 \begin{proposition}\label{p:Haus-conv-balls}
Let $\overline{\N}=\N\cup \{\infty\}.$ Let $\B$ be a finite-dimensional C*-algebra and let $(\tau^n)_{n \in \overline{\N}}$ be a sequence of faithful tracial states on $\B$ such that $(\tau^n)_{n \in \N}$ weak* converges to $\tau_\infty$.   For each $n \in\overline{\N}$, define $C_n=\{b \in \B: \|b\|_{\tau^n}\leq 1\}$.

It holds that $(C_n)_{n \in \N}$ converges to $C_\infty$ in the Hausdorff distance with respect to $\|\cdot\|_\B$.
 \end{proposition}

\begin{proof}
Since $\B$ is finite dimensional there exist $N \in \N, m_1, m_2, \ldots, m_N\in \N$ and a *-isomorphism   $\alpha: \oplus_{k=1}^N M_{m_k}(\C)\rightarrow \B$ onto $\B$. Set $\oplus_{k=1}^N M_{n_k}(\C)=\A$. For each $n \in \overline{\N}$, define that $\sigma^n=\tau^n\circ \alpha $. We have that $(\sigma^n)_{n \in \N}$ is s sequence of faithful tracial states that weak* converges to $\sigma_\infty$.
    Let $n \in \overline{\N}$, since $\sigma^n$ is a faithful tracial state there exist $\mu^n_{1}, \mu^n_2, \ldots, \mu^n_N\in (0, \infty)$ such that $\sum_{k=1}^N \mu^n_k=1$ and 
    \[
    \sigma^n((a_1, a_2, \ldots, a_N))=\sum_{k=1}^N \frac{\mu^n_k}{m_k} \mathrm{Tr}(a_k)
    \]
    for every $(a_1, a_2, \ldots, a_N)$.  By weak* convergence, we have that $((\mu^n_1, \mu^n_2, \ldots, \mu^n_N))_{n \in \N}$ converges to $(\mu^\infty_1, \mu^\infty_2, \ldots, \mu^\infty_N)$ in the product topology on $\R^N$.
    
    Define $D_n=\{a \in \A: \|a\|_{\sigma^n}\leq 1\}$. Let $a \in D_\infty$. Now since $\sigma^n$ is faithful we may define
    \[
    y=\left(\frac{\sqrt{\mu^\infty_1}}{\sqrt{\mu^n_1}} a_1, \frac{\sqrt{\mu^\infty_2}}{\sqrt{\mu^n_2}} a_2 , \ldots, \frac{\sqrt{\mu^\infty_N}}{\sqrt{\mu^n_N}} a_N\right).
    \]
    Thus
    \begin{align*}
        \|y\|_{\sigma^n}^2&= \sigma^n(y^*y)\\
        & = \sigma^n\left( \frac{\mu^\infty_1}{\mu^n_1}a_1^*a_1, \frac{\mu^\infty_2}{\mu^n_2}a_2^*a_2, \ldots ,\frac{\mu^\infty_N}{\mu^n_N}a_N^*a_N \right)\\
        & = \sum_{k=1}^N \frac{\mu^n_\infty}{m_k} \mathrm{Tr}(a_k^*a_k)\\
        & = \|a\|_{\sigma^\infty}^2\leq 1.
    \end{align*}
    Thus $y \in D_n$. Next,
    \begin{align*}
        \|a-y\|_\A& = \max\{\|a_1-y_1\|_{M_{n_1}(\C)}, \|a_2-y_2\|_{M_{n_2}(\C)}, \ldots, \|a_N-y_N\|_{M_{n_N}(\C)}\}
    \end{align*}
    Consider $k \in \{1,2,\ldots, N\}$. We have that since the operator norm is bounded by the Frobenius norm
    \begin{align*}
    \|a_k-y_k\|_{M_{n_k}(\C)}& = \left\| a_k- \frac{\sqrt{\mu^\infty_k}}{\sqrt{\mu^n_k}} a_k\right\|_{M_{n_k}(\C)}\\
    & = \left|1-  \frac{\sqrt{\mu^\infty_k}}{\sqrt{\mu^n_k}}  \right|\cdot \|a\|_{M_{n_k}(\C)}\\
    & \leq \left|1-  \frac{\sqrt{\mu^\infty_k}}{\sqrt{\mu^n_k}}  \right|\cdot  \sqrt{\mathrm{Tr}(a_k^*a_k)}.
    \end{align*}
However, as
\[
\sum_{k=1}^N \frac{\mu^\infty_k}{m_k} \mathrm{Tr}(a_k^*a_k)=\|a\|_{\sigma^\infty}^2\leq 1,
\]
we have that $ \frac{\mu^n_\infty}{m_k} \mathrm{Tr}(a_k^*a_k)\leq 1$, and so
\[
\sqrt{\mathrm{Tr}(a_k^*a_k)}\leq \frac{\sqrt{m_k}}{\sqrt{\mu^\infty_k}}
\]
and thus 
\[
\|a_k-y_k\|_{M_{n_k}(\C)}\leq \left|1-  \frac{\sqrt{\mu^\infty_k}}{\sqrt{\mu^n_k}}  \right|\cdot \frac{\sqrt{m_k}}{\sqrt{\mu^\infty_k}}
\]
Hence
\[
\|a-y\|_\A\leq \max \left\{ \left|1-  \frac{\sqrt{\mu^\infty_k}}{\sqrt{\mu^n_k}}  \right|\cdot \frac{\sqrt{m_k}}{\sqrt{\mu^\infty_k}}: k \in \{1,2,\ldots, N\} \right\}.
\]
By a symmetric argument, we have that 
\begin{align*}
Haus_{\|\cdot\|_\A}(D_n, D_\infty)&  \leq \max\Bigg\{\max \left\{ \left|1-  \frac{\sqrt{\mu^\infty_k}}{\sqrt{\mu^n_k}}  \right|\cdot \frac{\sqrt{m_k}}{\sqrt{\mu^\infty_k}}: k \in \{1,2,\ldots, N\} \right\}\ , \\
& \quad \quad \quad \quad \quad   \ \max \left\{ \left|1-  \frac{\sqrt{\mu^n_k}}{\sqrt{\mu^\infty_k}}  \right|\cdot \frac{\sqrt{m_k}}{\sqrt{\mu^n_k}}: k \in \{1,2,\ldots, N\} \right\} \Bigg\}
\end{align*}
 by definition of the Hausdorff distance. Thus as $(\mu^n_k)_{n \in \N}$ converges to $\mu^\infty_k$ for each $k \in \{1,2,\ldots, N\}$, we have that $\lim_{n \to \infty}Haus_{\|\cdot\|_\A}(D_n, D_\infty)=0.$ By construction of $\sigma^n$ and since $\alpha$ is a *-isomorphism, the proof is complete.
\end{proof}

We use these results  to provide continuous fields of $L$-seminorms.

\begin{theorem}\label{t:cont-field-L}
    Let $m \in \overline{\N}=\N\cup \{\infty\}$.  Let $(\tau^n)_{n \in \overline{\N}}$ be a sequence of faithful tracial states on $\A$.  If $(\tau^n_m)_{n \in \N}$ of Expression \eqref{eq:fd-trace} weak* converges to $\tau^\infty_m$ on $\B_m$, then for every $a \in \B_m$, we have $(L^{\tau^n_m}_\beta(a))_{n \in \N}$ of Expression \eqref{eq:ind-lim-fd-lip} converges to $L^{\tau^\infty_m}_\beta(a)$ in the usual topology on $\R$.
\end{theorem}
\begin{proof}
Let $a \in \B_m$. 
   Let $n \in \overline{\N}$, define
    \[
    f_n: \B_m \rightarrow \R
    \]
    by $f_n(b)=\|[D^{\tau^n_m}_\beta,\pi_{\tau^n_m}(a) ](b)\|_{\tau^n_m}$. Note that $f_n$ is continuous with respect to $\|\cdot\|_{\B_m}$ by finite dimensionality. 

Define
\[
C_n=\{b \in \B_n: \|b\|_{\tau^n_m}\leq 1\},
\]
which is compact with respect to $\|\cdot\|_{\B_m}$ by finite dimensionality.

Next, we verify that all the $C_n$'s are contained in one compact set. By finite dimensional, there exists a sharp $\nu_n>0$   such that 
\[
\|\cdot\|_{\B_m} \leq \nu_n\cdot \|\cdot \|_{\tau^n_m}.
\]
By \cite[Proposition 3.10]{Aguilar24}, we have that $(\nu_n)_{n \in \N}$ converges to $\nu_\infty$. Hence $r=\sup_{n \in \overline{\N}}  \nu_n<\infty $. Now, let $b \in C_n$, then $\|\cdot \|_{\tau^n_m}\leq 1$, and so
\[
\|b\|_{\B_m} \leq \nu_n\cdot \|\cdot \|_{\tau^n_m}\leq \nu_n \leq r.
\]
Hence $b \in \{b \in \B_m : \|b\|_{\B_m} \leq r\}$. Set $C'=\{b \in \B_m : \|b\|_{\B_m} \leq r\}$. We have that $C'$ is compact by finite dimensionality and that $C_n \subseteq C'$ for every $n \in \overline{N}$.

Next, by Proposition \ref{p:constant-conv} and by a similar argument to \cite[Proposition 3.6]{Aguilar24}, we have that $(f_n)_{n \in \N}$ converges uniformly to $f_\infty$ on any compact subset of $(\B_m, \|\cdot\|_{\B_m})$ including $C'$. Finally, we have that $(C_n)_{n \in \N}$ converges to $C_\infty$ in the Hausdorff distance with respect to $\|\cdot\|_{\B_m}$ by weak* convergence  by Proposition \ref{p:Haus-conv-balls}. Therefore by Lemma \ref{l:sup-conv}, we have that 
\[
(\sup_{b \in C_n}f_n(b))_{n\in \N}=(L^{\tau^n_m}_\beta(a))_{n \in \N}
\]
 converges to $\sup_{b \in C_\infty} f_\infty (b)=L^{\tau^\infty_m}_\beta(a) $ in the usual topology on $\R.$
\end{proof}

\section{Convergence of sequences of Effros-Shen algebras and UHF algebras}

We will now provide our main convergence results. But first, we need notation for each of these applications. We begin with the Effros-Shen algebras which were first defined in 
\cite{Effros80b}.

 Let $\theta \in \R$ be irrational. There exists a unique sequence of   integers $(r^\theta_n)_{n \in \N }$  with $r^\theta_n>0$ for all $n \in \N\setminus \{0\}$ such that 
 \[
 \theta =\lim_{n \to \infty} r_0^\theta +\cfrac{1}{r^\theta_1 + \cfrac{1}{r^\theta_2 + \cfrac{1}{r^\theta_3 +\cfrac{1}{\ddots+\cfrac{1}{r^\theta_n}}}}}.
 \]
 When $\theta \in (0,1)$, we have that $r^\theta_0=0$. The sequence $(r^\theta_n)_{n \in \N_0}$ is   the {\em continued fraction expansion of $\theta$} \cite{Hardy38}. 
  
 For each $n \in \N$, define \[
p_0^\theta=r_0^\theta, \quad p_1^\theta=1 \quad \text{ and } \quad q_0^\theta=1, \quad q_1^\theta=r^\theta_1,
\]
and   set 
\[
p_{n+1}^\theta=r^\theta_{n+1} p_{n}^\theta+p_{n-1}^\theta
\]
and 
\[
q_{n+1}^\theta= r^\theta_{n+1} q_{n}^\theta+q_{n-1}^\theta.
\]
The sequence $\left(p_{n}^\theta/q_{n}^\theta\right)_{n \in \mathbb{N}_0}$ of {\em convergents} $p^\theta_n/q^\theta_n$ converges to $\theta$. In fact, for each $n \in \N$, 
\[
 \frac{p_n^\theta}{q_n^\theta}=r_0^\theta +\cfrac{1}{r^\theta_1 + \cfrac{1}{r^\theta_2 + \cfrac{1}{r^\theta_3 +\cfrac{1}{\ddots+\cfrac{1}{r^\theta_n}}}}}.
\]
 
 We now define the terms for the inductive sequence that form the Effros-Shen algebras. Let $\B_{\theta,0}=\mathbb{C}$ and, for each $n \in \mathbb{N}_0$, let
\[
\B_{\theta,n}=M_{q_n^\theta}(\C) \oplus M_{q_{n-1}^\theta}(\C)
\]
and for each $n \in \N$, set $\A_{\theta, n}=\alpha^{(n)}(\B_{\theta,n})$.

These form an inductive sequence with the   maps   
 \begin{equation}\label{eq:theta-alg}
\alpha_{\theta,n}:a\oplus b \in \B_{\theta,n} \mapsto   \diag\left(a, \ldots, a,b \right)  \oplus a \in \B_{\theta,n+1},
\end{equation}
where there are $r^\theta_{n+1}$ copies of $a$ on the diagonal in the first summand of $\B_{\theta,n+1}$. This   is a unital *-monomorphism by construction.  For $n=0$,
\[\alpha_{\theta, 0}: \lambda \in  \B_{\theta,0} \mapsto   \diag(\lambda, \ldots, \lambda)  \oplus \lambda\ \in \B_{\theta,1}.
\]
The {\em Effros--Shen algebra associated to $\theta$} is   the inductive limit (see \cite[Section 6.1]{Murphy90})
\[
\A_\theta=\underrightarrow{\lim} \ (\B_{\theta,n}, \alpha_{\theta,n})_{n \in \N}
\]
by \cite{Effros80b}.

There exists a unique faithful tracial state $\tau^\theta$ on $\A_\theta$ such that for each $n \in \N \setminus \{0\}$,  $\tau_{\theta,n}$ (see Expression \eqref{eq:fd-trace}) is defined for each $(a,b) \in \B_{\theta,n}$ by
\[
\tau_{\theta,n}(a,b)=t(\theta,n)\frac{1}{q_n^\theta}\mathrm{Tr}(a)+(1-t(\theta,n))\frac{1}{q_{n-1}^\theta}\mathrm{Tr}(b),
\]
where   \[t(\theta,n)=(-1)^{n-1}q_n^\theta  (\theta q_{n-1}^\theta -p_{n-1}^\theta ) \in (0,1)\]
 (see \cite[Lemma 5.5]{Aguilar-Latremoliere15}).
 
For each $n \in \N$, define  
 \begin{equation}\label{eq:beta-theta}
      \beta^\theta_n=\frac{1}{\dim(\A_{\theta,n})}=\frac{1}{(q_n^\theta)^2+(q_{n-1}^\theta)^2},
 \end{equation}
and note that $(\beta^\theta_n)_{n \in \N}$ is summable by \cite{Hardy38}. Finally, for each $n \in \N$, define 
\[
a^{\tau_\theta}_n=\frac{c^{\tau^\theta}_n}{\beta^\theta_n}
\]
where $c^{\tau^\theta}_n$ is given by Expression \eqref{eq:sharp-constants}.
\begin{theorem}\label{t:es}
The map
\[
\theta \in (0,1)\setminus \Q \longmapsto (\A_\theta, L^{\tau_\theta}_{\beta_\theta})
\]
is continuous with respect to the quantum  Gromov-Hausdorff propinquity of \cite{Latremoliere13} where $L^{\tau_\theta}_{\beta_\theta}$ is given by Theorem \ref{t:CI-thm}.
\end{theorem}
\begin{proof}
Note that for every $\theta\in (0,1)\setminus \Q$ there exists a summable sequence of positive reals $(\beta_n)_{n \in\N}$ such that $\beta^\theta_n\leq \beta_n$ for every $n \in \N$ (see the beginning of the proof of \cite[Theorem 5.14]{Aguilar-Latremoliere15}. Now,
    let $(\theta^n)_{n \in \N}$ be a sequence in $(0,1)\setminus \Q$ that converges to some $\theta \in (0,1)\setminus \Q $ with respect to the usual topology on $\R$. Let $\varepsilon>0$. Choose $N_1 \in \N$  
    such that $\sum_{k=n}^\infty \beta_n < \varepsilon/3$ for every $n \geq N_1$. Now choose $N_2 \in \N$ such that $N_2\geq N_1$ and $q^{\theta_n}_k=q^\theta_k$ for every $n \geq N_1$ and $k \in \{1,2,\ldots, N_1\}$ which is possible by \cite[Proposition 5.10]{Aguilar-Latremoliere15}. Thus, for every $n \geq N_2$, we have $\B_{\theta_{n}, k}=\B_{\theta, k}$ and $\alpha_{\theta_n, k}=\alpha_{\theta, k}$ for every $k \leq N_1$. Now by \cite[Lemma 5.5]{Aguilar-Latremoliere15}, we have that $(\tau_{\theta_l, N_1})_{l \geq N_2}$ converges to $\tau_{\theta,N_1} $  in the weak* topology. Thus, by the same proof as \cite[Lemma 5.13]{Aguilar-Latremoliere15}, we have that there exists $N_3\geq N_2$ such that 
    \[
    \qpropinquity{} ((\B_{N_1}, L^{\tau_{\theta_n, N_1}}_{\beta^\theta_n}), (\B_{N_1}, L^{\tau_{\theta, N_1}}_{\beta_{\theta}}))< \varepsilon/3
    \]
    by Theorem \ref{t:cont-field-L}. Let $n \geq N_3$. By Theorem \ref{t:fd-approx} and Theorem \ref{t:ind-lim-comm} and the triangle inequality, we have
    \begin{align*}
     & \qpropinquity{} ((\A_{\theta_n}, L^{\tau^{\theta_n}}_{\beta^{\theta_n}}), (\A_{\theta}, L^{\tau^{\theta}}_{\beta^\theta}))\\
     & \quad \leq  \qpropinquity{} ((\A_{\theta_n}, L^{\tau^{\theta_n}}_{\beta^{\theta_n}}), (\A_{\theta_n,N_1}, L^{\tau^{\theta_n}}_{\beta^{\theta_n}}))  \\
     & \quad \quad +  \qpropinquity{} ( (\A_{\theta_n,N_1}, L^{\tau^{\theta_n}}_{\beta^{\theta_n}}), (\A_{\theta_n,N_1}, L^{\tau^{\theta}}_{\beta^\theta}))\\
     & \quad \quad \quad + \qpropinquity{}((\A_{\theta_n,N_1}, L^{\tau^{\theta}}_{\beta^\theta}), (\A_{\theta}, L^{\tau^{\theta}}_{\beta^\theta}))\\
     & \leq \sum_{k=N_1}^\infty \beta^{\theta_n}_k + \qpropinquity{} ( (\A_{\theta_n,N_1}, L^{\tau^{\theta_n}}_{\beta^{\theta_n}}), (\A_{\theta_n,N_1}, L^{\tau^{\theta}}_{\beta^\theta}))+\sum_{k=N_1}^\infty \beta^{\theta_n}_k\\
     & \leq \sum_{k=N_1}^\infty \beta_k + \qpropinquity{} ( (\A_{\theta_n,N_1}, L^{\tau^{\theta_n}}_{\beta^{\theta_n}}), (\A_{\theta_n,N_1}, L^{\tau^{\theta}}_{\beta^\theta}))+\sum_{k=N_1}^\infty \beta_k\\
     & < \varepsilon/3+\qpropinquity{} ( (\A_{\theta_n,N_1}, L^{\tau^{\theta_n}}_{\beta^{\theta_n}}), (\A_{\theta_n,N_1}, L^{\tau^{\theta}}_{\beta^\theta}))+\varepsilon/3\\
     & = 2\varepsilon/3+ \qpropinquity{} ((\B_{N_1}, L^{\tau_{\theta_n, N_1}}_{\beta^\theta_n}), (\B_{N_1}, L^{\tau_{\theta, N_1}}_{\beta_{\theta}}))\\
     & < 2\varepsilon/3+\varepsilon/3=\varepsilon
    \end{align*}
    as desired.
\end{proof}

  Next, we move to the UHF case. 
\begin{definition}\label{Baire-Space-def}
The \emph{Baire space} $\BaireSpace$ is the set $(\N\setminus\{0\})^\N$ endowed with the metric $\mathsf{d}$ defined, for any two $(x(n))_{n\in\N}$, $(y(n))_{n\in\N}$ in $\BaireSpace$, by 
\begin{equation*}
 d_\BaireSpace\left((x(n))_{n\in\N}, (y(n))_{n\in\N}\right) = \begin{cases}
0  \ \ \ \ \text{ if $x(n) = y(n)$ for all $n\in\N$},\\ \\
2^{-\min\left\{ n \in \N : x(n) \not= y(n) \right\}} \ \ \  \text{ otherwise}\text{.}
\end{cases}
\end{equation*}
\end{definition}

Next, we define UHF algebras in a way that  suits  our needs. Given    $(\beta(n))_{n\in\N} \in \BaireSpace$, let 
\[
\boxtimes \beta(n)=\begin{cases}
1 & \text{ if } n=0,\\
\prod_{j=0}^{n-1} (\beta(j)+1) & \text{ otherwise}.
\end{cases}
\]

For each $n \in \N$, define a unital *-monomorphism by
\[
\mu_{\beta,n} : a \in \M_{\boxtimes \beta(n)}(\C) \longmapsto \mathrm{diag}(a,a,\ldots, a) \in \M_{\boxtimes \beta(n+1)}(\C),
\]
where there are $\beta(n)+1$ copies of $a$ in $\mathrm{diag}(a,a,\ldots,a)$. Set $\mathsf{uhf}((\beta(n))_{n\in\N})=\underrightarrow{\lim} \ (\M_{\boxtimes \beta(n)}(\C) , \mu_{\beta,n})_{n \in \N}$. The map
\[
(\beta(n))_{n\in\N} \in \BaireSpace \longmapsto \mathsf{uhf}((\beta(n))_{n\in\N})
\]
is a surjection onto the class of all UHF algebras up to *-isomorphism by \cite[Chapter III.5]{Davidson}.

For each $n \in \N$, let 
\[
\gamma_\beta(n)=\frac{1}{\dim(\M_{\boxtimes \beta(n)}(\C))},
\]
and let \[\rho_\beta\] be the unique faithful tracial state on $uhf((\beta(n))_{n\in\N})$. We  now state our result for continuity of UHF algebras with respect to the Baire space.

\begin{theorem}
    The map
    \[
    \beta \in \BaireSpace \longmapsto (\mathsf{uhf}(\beta), L^{\rho_\beta}_{\gamma_\beta})
    \]
    is continuous with respect to the quantum  Gromov-Hausdorff propinquity of \cite{Latremoliere13} where $L^{\rho_\beta}_{\gamma_\beta}$ is given by Theorem \ref{t:CI-thm}.
\end{theorem}
\begin{proof} 
This follows similarly as the proof of Theorem \ref{t:es} since convergence in the Baire space is equivalent to convergence of irrationals by \cite[Proposition 5.10]{Aguilar-Latremoliere15}.
\end{proof}

\bibliographystyle{amsplain}
\bibliography{thesis}
\vfill

\end{document}